\documentclass[a4paper, 11pt]{article}

\usepackage[english]{babel}
\usepackage[utf8x]{inputenc}
\usepackage{ulem}
\usepackage[T1]{fontenc}
\usepackage{subfiles}

\usepackage{amsmath, amsthm, amsfonts, amssymb, bbm, algorithm, algorithmic,mathrsfs,dsfont, stmaryrd}
\usepackage{cancel}
\usepackage{xr-hyper} 
\usepackage[table,dvipsnames]{xcolor}
\usepackage{graphicx, caption, subcaption, natbib, enumitem, hyperref, minitoc, pgf,tikz,makecell, rotating}
\usepackage[export]{adjustbox}

\usetikzlibrary{shapes.geometric,arrows,mindmap,fit,patterns,decorations.pathreplacing}

\usepackage[top=3cm, bottom=3cm, left=3cm, right=2.5cm]{geometry}

\usepackage[linecolor=blue!60!,backgroundcolor=blue!10!,textwidth=2.1cm,textsize=tiny]{todonotes}

\newtheorem{theorem}{Theorem}

\newtheorem{lemma}[theorem]{Lemma}

\newtheorem{assumption}{Assumption}

\theoremstyle{remark}
\newtheorem{remark}{Remark}

\DeclareMathOperator{\argmin}{argmin}
\DeclareMathOperator{\ind}{\mathbbm{1}}

\DeclareMathOperator{\E}{\mathbb{E}}
\DeclareMathOperator{\R}{\mathbb{R}}

\DeclareMathOperator{\FDR}{FDR}
\DeclareMathOperator{\FDP}{FDP}
\DeclareMathOperator{\TDR}{TDR}
\DeclareMathOperator{\TDP}{TDP}

\newcommand{\wt}[1]{{\widetilde{#1}}}
\newcommand{\wh}[1]{{\widehat{#1}}}

\DeclareMathOperator{\DtrO}{\cD^0_{\text{train}}}
\DeclareMathOperator{\Dcal}{\cD_{\text{cal}}}
\newcommand{\Dcall}{\cD_{\text{cal}}}
\DeclareMathOperator{\Dtest}{\cD_{\text{test}}}
\newcommand{\Dtestt}{\cD_{\text{test}}}

\newcommand{\Ztr}{Z_{\text{train}}}
\DeclareMathOperator{\Otr}{\Omega_{\text{train}}}

\DeclareMathOperator{\cDNull}{\cD^0}
\DeclareMathOperator{\cDAlt}{\cD^1}

\newcommand{\EM}{\ensuremath}

\newcommand{\cB}{\EM{\mathcal{B}}}

\newcommand{\cD}{\EM{\mathcal{D}}}
\newcommand{\cE}{\EM{\mathcal{E}}}

\newcommand{\cG}{\EM{\mathcal{G}}}
\newcommand{\cH}{\EM{\mathcal{H}}}
\newcommand{\cI}{\EM{\mathcal{I}}}

\newcommand{\cL}{\EM{\mathcal{L}}}

\newcommand{\cZ}{\EM{\mathcal{Z}}}

\newcommand{\cpink}[1]{\textcolor{purple}{#1}}

\definecolor{mygreen}{rgb}{0.82, 1.0, 0.82}
\definecolor{myred}{rgb}{ 1.0, 0.84, 0.84}

\graphicspath{ {./img/} }

\tikzset{
    side by side/.style 2 args={
        line width=2pt,
        #1,
        postaction={
            clip,postaction={draw,#2}
        }
    }
    }

\renewcommand*{\thefootnote}{\fnsymbol{footnote}}

\begin{document}
\begin{center}
{\LARGE
	{  {FDR control and FDP bounds for conformal link prediction}}
}
\bigskip

Gilles Blanchard$^{1}$, Guillermo Durand$^{1}$, Ariane Marandon$^{2}$, Romain P\'erier$^{1}$%\footnotemark{},  
\bigskip

{\small
{
$^{1}$ 
LMO, Universit\'e Paris-Saclay, gilles.blanchard@universite-paris-saclay.fr, guillermo.durand@universite-paris-saclay.fr, romain.perier@universite-paris-saclay.fr

$^{2}$ 
Turing Institute, amarandon-carlhian@turing.ac.uk

}
}
\bigskip

%\footnotetext{Corresponding author: ariane.marandon-carlhian@sorbonne-universite.fr.}

\end{center}
\bigskip

\begin{abstract}
{
In \cite{marandon2023conformal}, the author introduces a procedure to detect true edges from a partially observed graph using a conformal prediction fashion: first computing scores from a trained function, deriving conformal $p$-values from them and finally applying a multiple testing procedure. In this paper, we prove that the resulting procedure indeed controls the FDR, and we also derive uniform FDP bounds, thanks to an exchangeability argument and the previous work of \cite{marandon22a}.
}

\end{abstract}

%{\bf Keywords.} Conformal inference, false discovery rate, link prediction. 
\renewcommand*{\thefootnote}{\arabic{footnote}}
\setcounter{footnote}{0}

\section{Introduction}

{
Consider a link prediction setting where we partially observe a graph, that is, the presence or absence of an edge (borrowing to multiple testing literature, we will use the terms ``true'' and ``false'' edge, respectively) between two nodes is either observed, either masked \citep{lu2011link}. Furthermore, consider that the missing data mask is known. Hence, the absence of connection between two nodes is either an observed false edge, either a masked edge (and the state of the edge, true or false, is unobserved), and we know if we are in the first or the second case. The task at hand is to provide a label (true or false) to each unobserved edge. In \citet{marandon2023conformal}, the author introduced the first, to our knowledge, method which aims to control the False Discovery Rate (FDR) in the given context, using tools from the Conformal Inference (CI) literature, namely conformal $p$-values \citep{MR2161220}. However, because of the graph structure inducing intricate dependence in the data, the setup is different from the usual case encountered in CI: typically, in CI, the data is assumed to be exchangeable. Hence, theoretical FDR control was not proven, but was observed empirically on simulated data.
}

{
In this work, we prove that the procedure of \citet{marandon2023conformal} controls the FDR under a general missing data setup and without any distributional assumptions on the complete network. The core observation to make is that even if the data is not exchangeable, the scores are, under really mild assumptions (no ties among them, almost surely). This allows to use the results of \citet{marandon22a} to achieve FDR control. Furthermore, the exchangeability of the scores also allows to use \citet{GBR2023} to get a uniform False Discovery Proportion (FDP) bound for the path $(R(t))_{t\in[0,1]}$, where $R(t)$ is the set of unobserved edges with conformal $p$-value lesser or equal than $t$. This uniform bound in turn allows to construct confidence post hoc bounds \`a la \citet{MR4124323}.
}

{
For an extensive review of related works in link prediction, conformal inference, and multiple testing, we refer the reader to \citet[Sections 1.3, 3.1, 3.3]{marandon2023conformal}.
}

{
  The rest of the paper is organized as follows. In Section~\ref{rappel}, we first define formally the model, the problem and the notation, before recalling the conformal link prediction procedure introduced in \citet{marandon2023conformal}. Then, in Section~\ref{theory}, we state and prove our two main theorems. We conclude in Section~\ref{discussion}
  with a short discussion of the relation to recent literature, in particular the work of \citet{huang2023uncertainty}.
}
%\todo{Discussion discutant des papiers en lien eg Candes etc.}

%\cite{marandon2023conformal} proposed a method for FDR control in link prediction, based on tools from conformal inference. 
%Here, we establish a theoretical guarantee for this method, which holds under a general missing data setup and without any distributional assumptions on the complete network. 
%the graph structure induces intricate dependence in the data, and so the setup different from the usual setup in conformal inference, where data exchangeability is assumed, 

%\corange{short intro}

%\tableofcontents

\section{Conformal link prediction}\label{rappel}

%\paragraph{Problem setup} 
\subsection{Model and notation}

Let $A^* = (A^*_{i,j})_{1 \leq i,j \leq n}$ be the adjacency matrix of the true complete graph $\cG$, $X \in \R^{n \times d}$ a matrix of node covariates (if available), and $\Omega = (\Omega_{i,j})_{1 \leq i,j \leq n}$ the sampling matrix such that $\Omega_{i,j} = 1$ if the interaction status (true/false) of $(i,j)$ is observed, and $0$ otherwise. 
We denote by $A$ the observed adjacency matrix with $A_{i,j} = \Omega_{i,j} A^*_{i,j}$. Thus, $A_{i,j}=1$ indicates that there is an observed true edge between $i$ and $j$, whereas $A_{i,j}=0$ indicates either the observed lack of an edge or an unreported edge. The sampling matrix $\Omega$ is assumed to be observed, so that it is known which zero-entries $A_{i,j}=0$ correspond to observed false edges and which ones correspond to missing information. 
We denote by $P$ the joint distribution of $Z^* = (A^*, X, \Omega)$, 
%which belongs to a family of distributions $\mathcal{P}$, 
$Z$ the observation $(A, X, \Omega)$ and $\cZ$ the observation space. {$P$ belongs to a model $\mathcal P$.} We consider a general missing data setup in which the entries of $\Omega$ are independently generated as:
\begin{align*}
\Omega_{i,j} \vert A^*, X \sim \cB( w_0 \ind_{A^*_{i,j} =0} + w_1 \ind_{A^*_{i,j} =1} ), \quad 1 \leq i, j \leq n 
\end{align*}
for some unknown sampling rates $w_0, w_1$. This type of missing data setup is called double standard sampling \citep{chiquet20, sportisse2020imputation} and is a generalization of the case where the entries of $\Omega$ are i.i.d, that is more relevant for practical applications (see Section 2 in \citet{marandon2023conformal} and the references therein for a discussion of this assumption with regards to the literature). 
Let us introduce the following notations:
\begin{itemize}
\item We denote by $\Dtest(Z)=\{ (i,j) \; \colon \;  \Omega_{i,j} = 0\}$ the set of non-sampled (or missing) node pairs and by $\cD(Z)=\{ (i,j) \; \colon \;  \Omega_{i,j} = 1\}$ the set of sampled pairs, with $\cDNull = \{ (i,j) \in \cD \; \colon \; A^*_{i,j} = 0 \}$ the set of observed false edges and $\cDAlt = \{ (i,j) \in \cD \; \colon \; A^*_{i,j} = 1 \}$ the set of observed true edges. We refer to $\Dtest(Z)$ as the test set.% and to $\cD(Z)$ as the training set. 
\item We denote by $\cH_0 = \{(i,j) \; \colon \; \Omega_{i,j} = 0, A^*_{i,j} = 0 \} $ the (unobserved) set of false edges in the test set and $\cH_1 = \{(i,j) \; \colon \; \Omega_{i,j} = 0, A^*_{i,j} = 1 \}$ the (unobserved) set of true edges in the test set. 
\end{itemize}
In link prediction, one is interested in classifying the unobserved node pairs {of} $\Dtest$ into true edges and false edges, or in other words, selecting a set of unobserved node pairs to be declared as true edges, based on the observed graph structure. 
Define a selection procedure as a (measurable) function $R=R(Z)$ that returns a subset of $\Dtest$ corresponding to the indices $(i,j)$ where an edge is declared. 
For any such procedure $R$, the False Discovery Rate (FDR) of $R$ is defined as the average of the False Discovery Proportion (FDP) of $R$ under the model parameter $P\in \mathcal{P}$, 
that is, 
\begin{align*}
\FDR(R)&=\E_{Z^* \sim P} [\FDP(R)],\:\:\: \FDP(R)=\frac{\sum_{i\in \cH_0} \ind_{i\in R}}{1\vee |R|}.%\label{equFDRFDP}
% \E_{Z\sim P}[\FDP(P,R)],\:\:\: \FDP(P,R)=\frac{\sum_{i\in \cH_0} \ind{i\in R}}{1\vee |R|}.\label{equFDRFDP}
\end{align*}
{Similarly, the true discovery rate (TDR) is defined as the average of the true discovery proportion (TDP), that is,}
\begin{align*}
\TDR(R)&= \E_{Z^* \sim P} [\TDP(R)],\:\:\: 
\TDP(R)=\frac{\sum_{i\in \cH_1} \ind_{i\in R}}{1\vee \vert \cH_1 \vert }.%\label{equTDRTDP}
%\TDR(P,R)&= \E_{Z\sim P}[\TDP(P,R)],\:\:\: 
%\TDP(P,R)=\frac{\sum_{i\in \cH_1} \ind{i\in R}}{1\vee m_1(P)}.\label{equTDRTDP}
\end{align*}
The aim considered here is to build a procedure $R$ that controls the FDR while having a TDR (measuring the \textit{power} of the procedure) as large as possible. 

{
Notably, a well-known procedure controlling the FDR is the Benjamini-Hochberg procedure \citep{BH1995} which, given $m$ $p$-values $p_1,\dotsc,p_m$, sorts them: $0=p_{(0)}\leq p_{(1)}\leq\dotsb p_{(m)}$, defines $\hat k=\max\left\{k\in\llbracket 0,m\rrbracket : p_{(k)}\leq \alpha\frac km\right\}$, and rejects the $p$-values smaller than $\alpha \frac{\hat k}m$.
}

{
Let us emphasize that, in our setting, $\cH_0$ is random, which is not entirely classical in the multiple testing litterature, but it is also not unusual, and can be traced back to, at least, the mixture model presented in \citet{MR1946571}.
}

%\paragraph{Conformal link prediction \citep{marandon2023conformal}} 
\subsection{Conformal link prediction procedure}
%Let $g: \cZ \rightarrow \R^{n \times n}$ be a \textit{scoring} function, that takes as input an observation $z \in  \cZ$, 
%and returns a \textit{score} matrix $(S_{i,j})_{1 \leq i,j \leq n} \in \R^{n \times n}$, with $S_{i,j}$ estimating how likely it is that $i$ is connected to $j$.

{
We briefly recall the procedure which is described in details in \citet[Section 3.2]{marandon2023conformal}. Assume that we have at hand a score estimator function $g: \cZ \rightarrow \R^{n \times n}$ that takes as input an observation $z \in  \cZ$ and returns a score matrix $(S_{i,j})_{1 \leq i,j \leq n} \in \R^{n \times n}$ that scores the likeliness of a true edge between couples of nodes. The function $g$ corresponds to a link prediction algorithm that would be trained on the ``training'' dataset $z$. Any off-the-shelf link prediction algorithm can be used, up to a specification detailed in the next paragraph. Our training set has to be different from a ``calibration'' set on which we will compute scores and compare them to the scores computed on the test set to build the conformal $p$-values. We build this calibration set $ \Dcal$ of false edges by sampling uniformly $\ell=  |\Dcal|$ edges without replacement from the set of observed false edges $\cDNull$. The algorithm is then trained on $\Ztr=(A,X,\Otr)$ where $(\Otr)_{i,j} = 0$ if $(i, j) \in \Dcal$ and $\Omega_{i,j}$ otherwise. That is, formally, we simply compute $g(\Ztr)$. The scores $S_{i,j}=g(\Ztr)_{i,j}$ are then computed for all $(i,j)\in \Dcal\cup \Dtest$. Conformal $p$-values $(p_{i,j})_{(i,j) \in \Dtest}$ are then computed following
\begin{align} \label{lp:eq:confpvalues}
    p_{i,j} = \frac{1}{\ell +1} \left(1 + \sum_{(u,v) \in \Dcal} \ind_{ \{S_{i,j} \leq S_{u,v} \} }\right), 
    \quad (i,j) \in \Dtest 
\end{align}
and finally the BH procedure is applied to them. Those steps are recapped in Algorithm~\ref{lp:algo:main}.
}

{
Although $g$ can be any measurable function, we now discuss a general shape of $g$ mild and commonly used. Let us introduce, for a given $K \in \{1, \dots, n \}$, and any adjacency matrix $a$,
\begin{align*}% \label{lp:eq:Khop}
w(a)_{i,j} = (a_{i, \bullet}, a_{j, \bullet}, a^2_{i, \bullet}, a^2_{j, \bullet}, \dots, a^K_{i, \bullet}, a^K_{j, \bullet}), 
\end{align*}
where $a^k_{i, \bullet} = (a^k_{i, u})_{1\leq u \leq n}$ for $1 \leq k \leq K$, that is, $a^k_{i, \bullet} $ is the $i$-th row of $a$ to the power $k$. For $z=(a,x,\omega)\in\mathcal Z$, the commonly used form for $g(z)$ is then
\begin{align} \label{lp:eq:scorefn2}
g(z)_{i,j} &= h\big( (w(a)_{i,j}, x_{N(i,K),\bullet},x_{N(j,K),\bullet}) ; \{(w(a)_{u,v}, x_{N(u,K),\bullet},x_{N(v,K),\bullet}, a_{u,v}), \omega_{u,v}=1 \}\big) 
\end{align}
with $h(\cdot\,;\cdot)$ some real-valued measurable function, $N(i,K)$ the $K$-hop neighborhood of node $i$, and $x_{S, \bullet}$ is the submatrix $(x_{u,l})_{\substack{u\in S\\1\leq l\leq d}}$ for any subset $S\in \llbracket 1, n\rrbracket$.} Note that the second argument of $h$ means exactly that $g$ is trained only on the information provided by the non-missing edges. Also note that because $(\Otr)_{i,j}=0$ for all $(i,j) \in \Dcal$, $\Otr$ treats the edges of $\Dcal$ as missing. Hence $g(\Ztr)$ indeed is trained on a set of observed edges that is disjoint from the edges of the calibration set, thus avoiding overfitting. 

{
This formulation indeed is mild and encompasses numerous link prediction algorithms. For example the simple common neighbors heuristic \citep{lu2011link} given by $g(z)_{i,j} = a_{i, \bullet}^\top a_{j, \bullet}$.
The larger class of Empirical Risk Minimization (ERM) methods (like in \citet{bleakley07} or \citet{zhang2018link}) also generally fall into this form. Take a class $\mathcal F$ of measurable functions taking values in $[0,1]$, that can be thought of as binary classifiers. It could be a class of (graph) neural networks for example. Also let $\cL$ a measurable loss function, for example the cross-entropy loss function. Then let $\hat f \in \argmin_{f\in\mathcal F}\sum_{(u,v): \omega_{u,v}=1} \cL \left(f(w(a)_{u,v}),a_{u,v}  \right) $ and finally $h\big(\cdot\,; \{(w(a)_{u,v}, a_{u,v}), \omega_{u,v}=1 \}\big) =\hat f(\cdot)$. This ERM framework indeed fits Equation~\eqref{lp:eq:scorefn2}.
More examples are provided in \citet[Section 3.3]{marandon2023conformal}.
}

%\todo{eq 2 n'est qu'une forme possible, mais ça marche avec tout $g$ mesurable, à écrire, preuve du lemme 1 à modifier en conséquence}

% \todo{VÉRIFIER CE QUE JE VIENS D'ÉCRIRE, C'EST LA PARTIE OÙ J'AI TENTÉ DE RESSUSCITER UN MORCEAU DE LA DERNIÈRE VERSION DU MANUSCRIT OÙ LA PREUVE FIGURAIT SANS RIEN CASSER AVEC LA VERSION ACTUELLE, LA PREUVE UTILISANT L'ÉQUATION \eqref{lp:eq:scorefn2}}

\begin{algorithm}[t]
Input: {Observation $Z=(A,X,\Omega)$: (Observed) adjacency matrix $A$, node covariate matrix $X$, sampling matrix $\Omega$; LP algorithm $g$; 
sample size $\ell$ of the reference set}\\
Define $\cD=\{(i,j) : \Omega_{ij} = 1 \}$; $\Dtest=\{(i,j) : \Omega_{ij} = 0 \}$; $\cDNull = \{ (i,j) \in \cD: A_{ij} =0\}$\\
1. Sample $\Dcal$ of size $\ell$ uniformly without replacement from $\cDNull$ \\
2. Learn $g$ on ``masked'' observation $\Ztr=(A,X,\Otr)$, with $\Otr$ given by 
 \begin{align*}
(\Otr)_{i,j}= \begin{cases}
      0 & \text{if}\ (i,j) \in \Dcal,  \\
      \Omega_{i,j} & \text{otherwise.}
    \end{cases}
\end{align*}
3. For each $(i,j) \in \Dcal \cup \Dtest$, compute the score $S_{i,j} = g(\Ztr)_{i,j}$ \\
4. For each $(i,j) \in \Dtest$, compute the conformal $p$-value $p_{i,j}$ given by~\eqref{lp:eq:confpvalues}\\
5. Apply BH using $(p_{i,j})_{(i,j) \in \Dtest}$ as input $p$-values, 
providing a rejection set $R(Z) \subset \Dtest$. \\
Output: { $R(Z)$} 
  \caption{Conformal link prediction}\label{lp:algo:main}
\end{algorithm}

%basic idea: given scoring function, compute conformal p-value which measure the statistical significance of "A*_ij=1" 
%by using a set of false edges to act as a reference set  
%then apply a multiple testing proc for FDR control 
%scores are learned by masking observation to ensure good representativness of the references scores .. 
%cite bates 

\section{FDR {and FDP} control guarantees}\label{theory}

In this section, we establish that the procedure controls the FDR in finite samples {and the FDP with high probability.}
In the sequel, we
% assume without loss of generality that the r.v. $\DtrO = \Dtr \cap \cD^0$ is equal to $\cD^0 \backslash \Dcal$ (as opposed to being a strict subset of $\cD^0 \backslash \Dcal$), and
use the following notation: $k_0 = \vert \cD^0 \vert$, $m= \vert \Dtest \vert$, {$\DtrO = \cD^0 \setminus \Dcal$}.
%\todo{\cred{$\xcancel{k_1 = \vert \cD^1 \vert}$ (jamais utilisé)} \cred{$\xcancel{\ell = |\Dcal|  }$ (déjà défini)}}

We will consider the following assumptions:

\begin{assumption}  \label{lp:ass:noties} The {matrix} score ${g(\Ztr) =\,\,} (S_{i,j})_{1 \leq i,j \leq n}$ has no ties a.s. 
\end{assumption}

\begin{assumption}  \label{lp:calsize} The calibration set size $\ell$ is a function of $k_0=|\cD^0|$ and $\cD^1$ only.
\end{assumption}

The main result of this section is the following:

\begin{theorem} \label{lp:thm1} Consider the output $R$ of the procedure proposed in Algorithm~\ref{lp:algo:main} with a scoring function $g$ of the form \eqref{lp:eq:scorefn2}, and let Assumptions~\ref{lp:ass:noties} and~\ref{lp:calsize} hold true. Then, it holds that $\FDR(R) \leq \alpha.$
\end{theorem}

The proof relies on reformulating the proposed procedure as the one of \citet{marandon22a} applied to a suitable ordering %$S'_1, \dots, S'_{\vert \Dcal \vert + \vert \Dtest \vert}$ 
of the scores,
%$(S_{i,j})_{(i,j)\in \Dcal cup \Dtest}$. 
which is shown 
%(Lemma \ref{lp:lem:exch}) 
to satisfy the sufficient exchangeability condition in \citet{marandon22a}  for FDR control, conditionally on a certain random variable $W$. %see Assumption 2. therein

%\begin{lemma} \label{lp:lem:Smeasure}
%For a scoring function $g$ of the form \eqref{lp:eq:scorefn2}, $(S_{i,j})_{1 \leq i,j \leq n}$ is measurable with respect to $(\cD^1, \cD^0 \backslash \Dcal)$. 
% \end{lemma}

The following lemma states two key properties of the calibration scores $(S_{i,j})_{(i,j) \in \Dcal}$ and test scores $(S_{i,j})_{(i,j) \in \Dtest}$. %In words, these are that there exists a certain r.v. $W$ such that (a) conditionally on $W$, $\Dcal$ and $\cH_0$ are sampled in an exchangeable manner from a certain subset of $\{ (i,j), \; A^*_{i,j} = 0 \}$ and (b) scores are measurable with respect to $W$. 
\begin{lemma}  \label{lp:lem:ScoreProp}
Let $W$ denote the r.v. given by %$\cgreen{(k_0, k_1, m, A^*, \cD^1, \DtrO)}$. 
${(k_0, A^*, \cD^1, \DtrO, {X})}$, and let Assumption~\ref{lp:calsize} hold true.
%\cgreen{Assume that the calibration set size $\ell$ is a function of $k_0$ and $\cD^1$ only.}
Then: % it holds that 
\begin{itemize}
\item [(i)] $(S_{i,j})_{1 \leq i,j \leq n}$ is measurable with respect to $W$. 
\item [(ii)] Denote, for any nonempty finite set $C$ and $N \in \mathbb{N}^*$ with $N{\leq} |C|$, $\Xi(C, N)$ the distribution that corresponds to sampling uniformly and without replacement an $N$-sized subset of the set $C$. Then it holds:
\begin{align} \label{lp:eq:DcalCondDistrib1}
\Dcal \vert W &\sim\Xi (\cE^0 \setminus \DtrO, \ell),\\
  \cH_0 &= (\cE^0 \setminus \DtrO) \setminus  \Dcal, \label{lp:eq:DcalCondDistrib2}\\
  \Dtest &  =  (\llbracket 1,n\rrbracket^2 \setminus (\cD^1 \cup \DtrO))  \setminus \Dcal,
\end{align}
{where $\cE^0 = \{(i,j) {\in \llbracket1,n\rrbracket^2} \colon A^*_{i,j} = 0 \}$.}
\end{itemize}
\end{lemma}
\begin{proof} Point (i) %is immediate 
{comes} from {$(S_{i,j})_{1 \leq i,j \leq n}=g(A, X, \Otr)$ as, first, $g$ is measurable, second,} {the knowledge of $\cD^1$ is equivalent to the knowledge of $A$, and {third}, $\Otr$ can be rewritten as \[(\Otr)_{i,j}= \begin{cases}
      1 & \text{if}\ (i,j) \in \cD^1 \cup  \DtrO,  \\
     0& \text{otherwise,}
    \end{cases}\]which in turn} entails that $(S_{i,j})_{1 \leq i,j \leq n}$ is measurable with respect to $(\cD^1, \DtrO, {X})$. \\
  
  For (ii), %let us denote %$\cE_0$ the set of all edges and 
  %$N_0=|\cE_0|$ the total number of null edges.
  %; denote also $\cE^0=\{(i,j) \in \cE: A^*_{i,j} =0 \}$. 
  we reason conditionally to $(A^*,k_0,\cD^1)$ and assume those to be fixed. 
  %Therefore,   $n=|\cD|=N-m$ is also fixed. 
  Because $\Omega$ are i.i.d. on $\cE^0$ conditionally to $(A^*, \cpink{X})$, recalling $\cD^0=\{(i,j) \in \cE^0\,:\Omega_{i,j}=1\}$ {and recalling the definition of $\Dcal$ as a sub-sampling without replacement of $\ell$ elements from $\cD^0$}, it follows that %conditionally to $n=|\cD|$, it holds
  \begin{align*}
    \cD^0 | (A^*,k_0,\cD^1, \cpink{X}) &\sim \Xi(\cE^0,k_0) ; \\
    \Dcal | (A^*,k_0,\cD^1,\cD^0, \cpink{X}) & \sim \Xi(\cD^0,\ell);  \\
    \DtrO & = \cD^0 \setminus \Dcal; \\
    \cH_0 & = \cE^0 \setminus \cD^0; \\
    {\Dtest} & = \llbracket1,n\rrbracket^2 \setminus (\cD^0 \cup \cD^1). 
  \end{align*} 
%  \todo{$\xcancel{\cD_1}$ si on veut 1ère ligne, $\xcancel{k_0}$ si on veut 2ème ligne, ça fait des statements plus précis mais moins pédagogiques, en y repensant je pense qu'il faut les garder}
    It is easy to check that the above joint distribution of $(\cD^0,\Dcal,\DtrO,{\Dtest},\cH_0)$ conditional to $(A^*,k_0,\cD^1)$
    is equivalently represented as
    \begin{align*}
      \DtrO| (A^*,k_0,\cD^1, \cpink{X}) &\sim \Xi(\cE^0,k_0-\ell) ;\\
      \Dcal | (A^*,k_0,\cD^1,\DtrO, \cpink{X}) & \sim \Xi(\cE^0\setminus \DtrO, \ell) ; \\
      \cD^0 &= \Dcal \cup \DtrO ; \\
      \cH_0 & = (\cE^0 \setminus \DtrO) \setminus \Dcal; \\
      { \Dtest} & = (\llbracket1,n\rrbracket^2 \setminus ( \cD^1 \cup \DtrO)) \setminus \Dcal.
    \end{align*}  
    This implies the claim.

\end{proof}
The next lemma introduces %the sequence of r.v.s that is equal to 
a suitable (random) ordering of the scores in order to formally grant the sufficient score exchangeability condition of \citet{marandon22a}. % conditionally on $W$. 

  \begin{lemma} \label{lp:lem:exchbis}
  Use the same notation and assumptions as in Lemma~\ref{lp:lem:ScoreProp}, denote $m_0 = |\cH_0|$.\\
 For given $W$, let $\sigma_0$ be a purely random ordering of $\cE^0 \setminus \DtrO$ (i.e. a random element
  drawn uniformly from $\mathrm{Bij}(\{1,\ldots,\ell + m_0\},\cE^0 \setminus \DtrO)$) and $\sigma_1$ an arbitrary ordering of $\cH_1$
  (an element of $\mathrm{Bij}(\{1,\ldots,|\cH_1|\},\cH_1)$), e.g. lexical ordering.
  Let:
  \begin{align*} 
    \Dcall^{\sigma_0} &= \{ \sigma_0(i), i=1,\ldots,\ell\};  \\
    \cH_0^{\sigma_0} &= \{ \sigma_0(i), i=\ell+1 ,\ldots, \ell+m_0\};  \\
    \Dtestt^{\sigma_0}& %\cred{:= (\{1,\ldots,n\}^2  \setminus ( \cD^1 \cup \DtrO)) \setminus \Dcall^{\sigma_0};}\\
    = %\{ \sigma_0(i), i=\ell+1,\ldots,\ell+m_0\} 
    { \cH_0^{\sigma_0} } \cup \{ \sigma_1(i), i=1,\ldots,|\cH_1|\};\\
    S'_i &= S_{\sigma_0(i)},  i=1,\ldots,\ell+m_0; \\
   S'_{\ell+m_0 + j} &=  S_{\sigma_1(j)},  j=1,\ldots,|\cH_1|. 
  \end{align*}
  %\todo{GD : $ \cH_0^{\sigma_0}   \cup \{ \sigma_1(i), i=1,\ldots,|\cH_1|\}$ au lieu de $\{ \sigma_0(i), i=\ell+1,\ldots,\ell+m_0\} \cup \{ \sigma_1(i), i=1,\ldots,|\cH_1|\}$ ?}
%  \todo{GB: comme j'ai ajouté $\Dtest$, qui dépend aussi de $\sigma_1$,
%    on peut se demander si la notation $\Dtestt^{\sigma_0}$ est pertinente. D'autre part, seul $\sigma_0$ est aléatoire. GD : $\Dtestt^{\sigma_0,\sigma_1}$ ? je suis pour les notations lourdes si elles sont plus explicatives. Ou alors on garde exactement $\Dtestt^{\sigma_0}$ justement parce que seul $\sigma_0$ est aléatoire}
 %\todo{GD : j'aime pas la notation $:=$ et elle n'est pas utilisée dans le reste}
Then the distribution of~$(\Dcall^{\sigma_0},{\Dtestt^{\sigma_0}},\cH_0^{\sigma_0})$ conditionally to $W$ is identical to that of $(\Dcal,{\Dtest},\cH_0)$ given by~\eqref{lp:eq:DcalCondDistrib1}-\eqref{lp:eq:DcalCondDistrib2}, and the scores $(S'_i)_{1 \leq i \leq \ell+m_0}$ form an exchangeable family of
 random variables (conditionally to $W$, hence also unconditionally).
  \end{lemma}
  \begin{proof}
  Straightforward since a (set-valued) random variable $F \sim \Xi(C,N)$ can be represented by a set whose elements are
  $N$ successive draws from $C$ without replacement, in turn equivalent to (i) drawing uniformly an ordering of $C$; (ii) taking
  the $N$ first elements according to that order.
  \end{proof}

  \begin{proof}[Proof of Theorem~\ref{lp:thm1}]
    {Let us consider the same construction and notation as in Lemma~\ref{lp:lem:exchbis} and define additionally
      \[\wt{p}_{i,j} = \frac{1}{\ell +1} \left(1 + \sum_{(u,v) \in \Dcall^{\sigma_0}} \ind_{ \{S_{i,j} \leq S_{u,v} \} }\right), 
        \quad (i,j) \in \Dtestt^{\sigma_0}; \]
      observe that it is merely definition~\eqref{lp:eq:confpvalues} with $\Dcal,\Dtest$ replaced by $\Dcall^{\sigma_0},\Dtestt^{\sigma_0}$.}
     % \todo{GD : Du coup pour mieux suivre l'analogie je propose ${p}^{\sigma_0}_{i,j}$ ou ${p}^{\sigma_0,\sigma_1}_{i,j}$  au lieu de $\wt{p}_{i,j}$}

{      We recall the definition of the FDP of a rejection set $R$:
      \[
        \FDP(R,\cH_0) = \frac{\vert R \cap \cH_0\vert}{\vert R \vert \vee 1},
      \]
      where for precision of notation, we have specified explicitly the set of null hypotheses $\cH_0$ on the left-hand side, because it is random
      in our observation model.
    }

    {For a rejection procedure $R((p_{i,j}))$ only depending on the family of $p$-values, observe that we have
      \begin{equation} \label{eq:fdpeq}
        \FDP(R((p_{i,j})_{(i,j) \in \Dtest}),\cH_0) \overset{\mathcal L}{=}  \FDP(R((\wt{p}_{i,j})_{(i,j) \in \Dtestt^{\sigma_0}}),\cH_0^{\sigma_0}),
      \end{equation}
      where $\overset{\mathcal L}=$ represents equality in distribution. Namely, by inspection of the definition
      of the family $p$-values~\eqref{lp:eq:confpvalues}, notice that conditionally to $W$
      those only depend on 
%      \todo{``they only depend'' comme si on parlait des p-val mais c'est pas plutôt du FDP dont il est question là ? surtout pour la dépendance en $\cH_0$\\
%     RP : La dépendance en $\cH_0$ est superflue puisqu'on a déjà une dépendance en $\Dcal$, si ça confuse on pourrait la retirer et ajouter "and that $\cH_0$ only depends on $\Dcal$"? } 
     $(\Dcal,\Dtest)$ (since the scores are measurable with respect to $W$, see Lemma~\ref{lp:lem:ScoreProp}). As a consequence, the left-hand side of~\eqref{eq:fdpeq} only depends on $(\Dcal,\Dtest,\cH_0)$,
      and the right-hand-side is obtained by replacing those 
      by $(\Dcall^{\sigma_0},\Dtestt^{\sigma_0},\cH_0^{\sigma_0})$, which according to Lemma~\ref{lp:lem:exchbis}
      have the same joint distribution conditional to $W$. Hence~\eqref{eq:fdpeq} holds conditionally to $W$ (hence
      also unconditionally). We therefore focus on the right-hand side of~\eqref{eq:fdpeq} from now on.
    }
    
For any set of $p$-values $(p_i)_{i \in \cI}$, denote by $R_{BH}((p_i)_{i \in \cI}) \subseteq \cI$ the rejection set of the BH procedure; {it takes the form
\[
  R_{BH}((p_i)_{i \in \cI}) = \Big\{ i \in I: p_i \leq t_{BH}( \{p_i,i \in I\}) \Big\},
\]
where $t_{BH}$ is a threshold that depends only on the (multi)set of
$p$-values $\{p_i,i \in I\}$.
}
% Here, we have that $\FDR = \E \left[ 
% \FDP \left(
% R ( (p_{i,j})_{(i,j) \in \Dtest} ) 
% \right)
% \right]
% $, with the $p$-values given by equation \eqref{lp:eq:confpvalues}. 
Let us introduce the set of r.v. $p'_1, \dots, p'_m$ such that 
\begin{equation} \label{eq:pvaluesp}
p'_a = \frac{1}{\ell+1} \left(1 + \sum_{b=1}^\ell \ind_{S'_{{\ell+}a} \leq S'_{b}} \right), \quad 1 \leq a \leq m, 
\end{equation}
with the r.v. $(S'_a)_{1 \leq a \leq m}$ defined in Lemma~\ref{lp:lem:exchbis}.

{It can be readily checked that
  \begin{equation}\label{eq:switchp}
    p'_a = \wt{p}_{\sigma_0(\ell+a)}, \quad 1 \leq a \leq m_0; \quad p'_{m_0+a} = \wt{p}_{\sigma_1(a)}, \quad 1 \leq a \leq \vert \cH_1 \vert.  
  \end{equation}
  As a consequence, we have that 
\begin{align}
\FDP \left(
R_{BH} ( (\wt{p}_{i,j})_{(i,j) \in \Dtestt^{\sigma_0}} ),\cH_0^{\sigma_0}  \right)
  &= \frac{ \vert \cH_0^{\sigma_0} \cap R_{BH}( (\wt{p}_{i,j})_{(i,j) \in \Dtestt^{\sigma_0}}) \vert}{\vert R_{BH}((\wt{p}_{i,j})_{(i,j) \in \Dtestt^{\sigma_0}}) \vert \vee 1}\notag \\
%  &\cred{\overset{\mathcal L}
%    = \frac{ \vert{\cH_0^{\sigma_0} \cap R( [p'])}\vert}{\vert{R[p']}\vert \vee 1}}, \label{lp:eq:fdp}
 & = \frac{ 
 \sum_{(i,j) \in \cH_0^{\sigma_0} } \ind \{ 
 \wt{p}_{i,j} \leq t_{BH} (\{ p_{i,j}, (i,j) \in \Dtestt^{\sigma_0}\})  \} 
 }{  \vert R_{BH}((\wt{p}_{i,j})_{(i,j) \in \Dtestt^{\sigma_0}}) \vert \vee 1} \nonumber \\
 &= \frac{ \sum_{a=1}^{m_0} \ind \{ 
   p'_a \leq t_{BH}( \{ p'_{b}, 1 \leq b \leq m\}) \}}{ \vert R_{BH} ( (p'_{b})_{1 \leq b \leq m} )  \vert \vee 1 }\label{lp:trucsigma} \\
  & = \FDP \left(
    R_{BH} ( (p'_a)_{1 \leq a \leq m} ) , \{1,\ldots,m_0\}
    \right),  \label{lp:eq:fdp}
\end{align}
where~\eqref{lp:trucsigma} is due to~\eqref{eq:switchp} and the definitions in Lemma~\ref{lp:lem:exchbis}.
}

Under Assumption \ref{lp:ass:noties}, %Lemmas~\ref{lp:lem:ScoreProp}-\ref{lp:lem:exchbis} combined with
Theorem 3.3. of \citet{marandon22a} entails that conditionally on $W$, the r.v. $(p'_a)_{1 \leq a \leq m_0}$  are each marginally super-uniform and the r.v. $(p'_a)_{1 \leq a \leq m }$ are PRDS on $\{1, \dots, m_0 \}$. It follows from \citet{BY2001} that the conditional expectation with respect to $W$ of the right term in \eqref{lp:eq:fdp} is below $\alpha$ (thus this holds also for the unconditional expectation). This concludes the proof.

\end{proof}

Furthermore, the same arguments as above (representing the distributions of null scores as a conditionally exchangeable distribution) also allow us to leverage
  recent results of \citet{GBR2023} to get a control of the false discovery proportion (FDP) with high probability, and which hold uniformly over the choice of the
  rejection threshold.
  \begin{theorem} \label{th:appgazin}
Consider the same setting and assumptions as in Theorem~\ref{lp:thm1}, but instead of applying the BH procedure as the last step of Algorithm~\ref{lp:algo:main},
    denote $R(t)$ as the rejection set obtained as the set of edges having conformal $p$-value less than $t$, for $t\in[0,1]$.\\    
    For any $\delta \in(0,1)$, with probability at least $1-\delta$ with respect to the sampling matrix $\Omega$ and the draw of $\Dcal$, it holds
    \begin{equation} \label{eq:probjoint}
      P\left(\forall t \in [0,1]: \FDP(R(t)) \leq \frac{m}{1\vee |R(t)|} (1+ \lambda(m,\ell,\delta))\right) \geq 1-\delta,
    \end{equation}
    where $\lambda(m,\ell,\delta)$ is the $1-\delta$ quantile of a certain universal probability
    distribution $P_{m,\ell}$, and satisfies
    \begin{equation} \label{eq:estlambda}
      \lambda(m,\ell,\delta) \leq \left( \frac{\log \delta^{-1} + \log(1 + 2\sqrt{\pi} (m \wedge \ell) ) }{m \wedge \ell} \right)^{\frac{1}{2}}.
    \end{equation}
  \end{theorem}
  \begin{proof}{ Following the same argument as in the proof of Theorem~\ref{lp:thm1}, we use the fact that
    \[
      ( (p_{i,j})_{j \in \Dtest} , \cH_0) \overset{\mathcal L}{=}
      ( (\wt{p}_{i,j})_{j \in \Dtestt^{\sigma_0}} , \cH_0^{\sigma_0}).
    \]
    Hence, since the event considered in the claim~\eqref{eq:probjoint} only depends on $( (p_{i,j})_{j \in \Dtest} , \cH_0)$,
    its probability is unchanged when those are replaced by $( (\wt{p}_{i,j})_{j \in \Dtestt^{\sigma_0}} , \cH_0^{\sigma_0})$.}
    
{Consider now the $p$-values $(p'_i)_{1 \leq i \leq m}$ given by
    explicit reindexing of $(\wt{p}_{ij})_{j \in \Dtestt^{\sigma_0}}$ via~\eqref{eq:pvaluesp}-\eqref{eq:switchp};
    again, since the event in~\eqref{eq:probjoint} only depends on the $p$-values and on $\cH_0$ as (multi)sets
    and not of the ordering of their indexation, we can use the reindexing given
    by $(p'_i)_{1 \leq i \leq m}$ (and the corresponding $\cH_0'=\{1,\ldots,m_0\}$) without changing the probability.
    Since the reindexed $p$-values $(p'_i)_{1 \leq i \leq m}$ can now be seen as conformal $p$-values based
    on exchangeable scores without ties $(S'_i)_{1 \leq i \leq \ell+m}$,}
%
%    %
%
 %   the scores
 %   are, conditionally to $W$, exchangeable and without ties.
    the claimed bound then
    holds, conditionally to $W$, as a direct application of Corollary~4.1 and Theorem~2.4
    of \citet{GBR2023} (with some straightforward estimates to simplify slightly the
    final expression). Thus the claimed probability bound also holds unconditionally.
      \end{proof}

  \begin{remark} Observe that the uniformity with respect to $t$ allows for a { post hoc}
    (i.e. data-dependent) choice of the rejection threshold, rather than having to set
    a target level $\alpha$ for the FDR in advance, {see \citet{MR4124323}}. This is of particular interest since
    the output of a link prediction procedure is often presented to the user as a list of edges ordered
    by their scores rather than a fixed rejection set; the user can then monitor by themselves
    the guaranteed FDP bound in dependence of the threshold and decide of an adequate trade-off
    based on possibly additional criteria (e.g. concerning the geometry of rejected edges,
    or using additional information).
    
    From the results of~\citet{GBR2023},
    the ``universal probability'' $P_{m,\ell}$ {appearing in Theorem~\ref{th:appgazin}} is that of the
    supremum deviation from identity of the empirical cdf for the color
    histogram of a P\'olya urn scheme with $\ell$ draws from an urn
    initially containing $(m+1)$ balls of different colors. It can be easily approximated
    by Monte-Carlo simulation to get a sharper estimate. The explicit estimate~\eqref{eq:estlambda}
    has the advantage of clearly showing an exponential concentration of order
    $\mathcal{O}( \sqrt{\log( ( m\wedge \ell)/\delta)/(m\wedge \ell)})$.
    The bound of \citet{GBR2023} also allows the replacement of $m$ by a specific built-in estimate
    $\wh{m}_0$ of $m_0$, and a slightly sharper estimate where the logarithmic
    dependence in $(n\wedge m)$ can be improved to an iterated logarithm. Both of
    these refinements are ignored here for simplicity.
  \end{remark}
  
  \section{Discussion}\label{discussion}
  
%  \todo{lien avec les 2 autres papiers, celui de Candès et celui concurrent de Gazin et al}
{
  \textbf{Comparison to \cite{huang2023uncertainty} :} Our arguments on score exchangeability %appears to be
  bears similarity to the % one from
  reasoning of \cite{huang2023uncertainty}, who considered conformal prediction intervals for label values on the nodes of a graph. A first (superficial) difference is that we predict on edges rather than nodes of the graph,
  and that we consider FDR and FDP for edge presence/absence rather than false coverage rate  prediction intervals.
  %for a  real-valued variable.
  % we want to underline that our results are different by applying this property in the context of FDR and FDP control in link prediction.
  We do not need an analogue of Assumption 1 of \cite{huang2023uncertainty}, which assumes explicitly that the score function is invariant
  with respect to permutation of indices of the calibration+test set. Namely,
  % of the ordering of the union of calibration and test sets. It appears unusual at first sight,
  this is in fact automatically satisfied in our setting, since the training function $g$ applied to
  the masked observation data $\Ztr$ in Algorithm~\ref{lp:algo:main} does not have the information of
  which edges are in $\Dcal$ and which in $\Dtest$.}

{More importantly, the nature of our results are different because, instead of fixing a threshold $t$ and get guarantees for the FDR or FDP of $R(t)$ as in \cite{huang2023uncertainty}, we consider either a (specific) data-dependant threshold $t_{BH}$ 
 %via BH procedure
   which ensures a control of the FDR, or a  bound uniform with respect to $t$ on the FDP which allows to use any data-dependent threshold and still ensures valid bounds on the FDP.  For a fixed beforehand threshold $t$, it is possible to follow the same
   line of argumentation we used to explicitly represent the scores as an exchangeable uple, and combine with the results
   of \cite{marquesf2023universal} instead of \cite{marandon2023conformal} or \cite{GBR2023}. In this case the same
   control as appearing in \cite{huang2023uncertainty} is recovered. See also \cite{GBR2023} for the relation
   of the distribution $P_{m,\ell}$ to the results of \cite{marquesf2023universal,huang2023uncertainty}.
   % \todo{Romain, à toi}
  %we do not use side data on edges to train score function, unlike in \cite{huang2023uncertainty} where node covariates of nodes from the calibration and test sets are used, and the ordering of this side data can affect the training.} 
}
\bibliographystyle{apalike}
\bibliography{bibli_lp}

\begin{thebibliography}{}

\bibitem[Benjamini and Hochberg, 1995]{BH1995}
Benjamini, Y. and Hochberg, Y. (1995).
\newblock Controlling the false discovery rate: a practical and powerful
  approach to multiple testing.
\newblock {\em J. Roy. Statist. Soc. Ser. B}, 57(1):289--300.

\bibitem[Benjamini and Yekutieli, 2001]{BY2001}
Benjamini, Y. and Yekutieli, D. (2001).
\newblock The control of the false discovery rate in multiple testing under
  dependency.
\newblock {\em Ann. Statist.}, 29(4):1165--1188.

\bibitem[Blanchard et~al., 2020]{MR4124323}
Blanchard, G., Neuvial, P., and Roquain, E. (2020).
\newblock Post hoc confidence bounds on false positives using reference
  families.
\newblock {\em Ann. Statist.}, 48(3):1281--1303.

\bibitem[Bleakley et~al., 2007]{bleakley07}
Bleakley, K., Biau, G., and Vert, J.-P. (2007).
\newblock {Supervised reconstruction of biological networks with local models}.
\newblock {\em Bioinformatics}, 23(13):i57--i65.

\bibitem[Efron et~al., 2001]{MR1946571}
Efron, B., Tibshirani, R., Storey, J.~D., and Tusher, V. (2001).
\newblock Empirical {B}ayes analysis of a microarray experiment.
\newblock {\em J. Amer. Statist. Assoc.}, 96(456):1151--1160.

\bibitem[Gazin et~al., 2023]{GBR2023}
Gazin, U., Blanchard, G., and Roquain, E. (2023).
\newblock Transductive conformal inference with adaptive scores.
\newblock arXiv preprint 2310.18108.

\bibitem[Huang et~al., 2023]{huang2023uncertainty}
Huang, K., Jin, Y., Candes, E., and Leskovec, J. (2023).
\newblock Uncertainty quantification over graph with conformalized graph neural
  networks.
\newblock In Oh, A., Neumann, T., Globerson, A., Saenko, K., Hardt, M., and
  Levine, S., editors, {\em Advances in Neural Information Processing Systems},
  volume~36, pages 26699--26721. Curran Associates, Inc.

\bibitem[Lu and Zhou, 2011]{lu2011link}
Lu, L. and Zhou, T. (2011).
\newblock Link prediction in complex networks: A survey.
\newblock {\em Physica A: statistical mechanics and its applications},
  390(6):1150--1170.

\bibitem[Marandon, 2023]{marandon2023conformal}
Marandon, A. (2023).
\newblock Conformal link prediction to control the error rate.
\newblock {\em arXiv preprint arXiv:2306.14693}.

\bibitem[Marandon et~al., 2022]{marandon22a}
Marandon, A., Rebafka, T., Roquain, E., and Sokolovska, N. (2022).
\newblock False clustering rate in mixture models.
\newblock {\em arXiv preprint arXiv:2203.02597}.

\bibitem[{Marques F.}, 2023]{marquesf2023universal}
{Marques F.}, P.~C. (2023).
\newblock On the universal distribution of the coverage in split conformal
  prediction.
\newblock {\em arXiv preprint 2303.02770}.

\bibitem[Sportisse et~al., 2020]{sportisse2020imputation}
Sportisse, A., Boyer, C., and Josse, J. (2020).
\newblock Imputation and low-rank estimation with missing not at random data.
\newblock {\em Statistics and Computing}, 30(6):1629--1643.

\bibitem[Tabouy et~al., 2020]{chiquet20}
Tabouy, T., Barbillon, P., and Chiquet, J. (2020).
\newblock Variational inference for stochastic block models from sampled data.
\newblock {\em Journal of the American Statistical Association},
  115(529):455--466.

\bibitem[Vovk et~al., 2005]{MR2161220}
Vovk, V., Gammerman, A., and Shafer, G. (2005).
\newblock {\em Algorithmic learning in a random world}.
\newblock Springer, New York.

\bibitem[Zhang and Chen, 2018]{zhang2018link}
Zhang, M. and Chen, Y. (2018).
\newblock Link prediction based on graph neural networks.
\newblock In {\em Advances in Neural Information Processing Systems}, pages
  5165--5175.

\end{thebibliography}

\section*{Acknowledgments}

A. Marandon acknowledges funding from the Turing-Roche Strategic Partnership.
GB, GD and RP acknowledge funding from the grants ANR-21-CE23-0035 (ASCAI), ANR-19-CHIA-0021-01 (BISCOTTE)
and  ANR-23-CE40-0018-01 (BACKUP) of the French National Research Agency ANR.
\end{document}